\documentclass[12pt,reqno]{article}

\usepackage[usenames]{color}
\usepackage{amssymb}
\usepackage{amsmath}
\usepackage{amsthm}
\usepackage{amsfonts}
\usepackage{amscd}
\usepackage{graphicx}

\usepackage[colorlinks=true,
linkcolor=webgreen,
filecolor=webbrown,
citecolor=webgreen]{hyperref}

\definecolor{webgreen}{rgb}{0,.5,0}
\definecolor{webbrown}{rgb}{.6,0,0}

\usepackage{color}
\usepackage{fullpage}
\usepackage{float}

\usepackage{graphics}
\usepackage{latexsym}
\usepackage{epsf}
\usepackage{breakurl}

\newcommand{\seqnum}[1]{\href{https://oeis.org/#1}{\rm \underline{#1}}}

\def\modd#1 #2{#1\ \mbox{\rm (mod}\ #2\mbox{\rm )}}
\def\tet{\uparrow\uparrow}

\begin{document}

\title{Brik's sequence:  a strange recursion}

\author{Jeffrey Shallit\\
School of Computer Science\\
University of Waterloo\\
Waterloo, ON  N2L 3G1 \\
Canada\\
\href{mailto:shallit@uwaterloo.ca}{\tt shallit@uwaterloo.ca}}

\maketitle

\theoremstyle{plain}
\newtheorem{theorem}{Theorem}
\newtheorem{corollary}[theorem]{Corollary}
\newtheorem{lemma}[theorem]{Lemma}
\newtheorem{proposition}[theorem]{Proposition}

\theoremstyle{definition}
\newtheorem{definition}[theorem]{Definition}
\newtheorem{example}[theorem]{Example}
\newtheorem{conjecture}[theorem]{Conjecture}

\theoremstyle{remark}
\newtheorem{remark}[theorem]{Remark}

\def\Enn{\mathbb{N}}

\begin{abstract}
We study the properties of the sequence of words $(B_i)$,
where $B_1 = 101$ and $B_{i+1} = B_i C_i$ for $i \geq 1$, where
$C_i$ is $B_i$ with the first $i$ symbols removed, and
the infinite binary sequence ${\bf b} = 10101101011011101 \cdots$
of which all the $B_i$ are prefixes.
We show that $\bf b$ is recurrent, but not uniformly recurrent; it has
exponential factor complexity; it is not morphic; and the density of $1$'s
exists and is transcendental.
\end{abstract}

\section{Introduction}

In January 2018, Garo Brik, then a student in one of the author's courses
at the University of Waterloo, proposed a very interesting iteration,
as follows:
$B_1 = 101$ and $B_{i+1} = B_i C_i$ for $i \geq 1$, where
$C_i$ is $B_i$ with the first $i$ symbols removed.
Thus, for example, we find
\begin{align*}
B_1 &= 101 &\quad C_1 &= 01 \\
B_2 &= 10101 &\quad C_2 &= 101  \\
B_3 &= 10101101 &\quad C_3 &= 01101 \\
B_4 &= 1010110101101 &\quad C_4 &= 110101101\\
B_5 &= 1010110101101110101101  &\quad C_5 &= 10101101110101101
\end{align*}
Since $B_{i-1}$ is a prefix of $B_i$ for all $i \geq 2$, there is a 
unique infinite sequence
$${\bf b} = 101011010110111010110110101101110101101010110111010110110101101110101101 \cdots$$
of which all the $B_i$ are prefixes.  Although at first glance
the iteration seems simple, and perhaps related to more familiar infinite
words such as the Thue-Morse word $\bf t$,
it turns out to have some subtle properties that are not evident at
first glance.  In this paper, we explore the properties of the $B_i$ and
$\bf b$, including factors, recurrence, 
and the frequency of the number of $1$'s.

The sequence $\bf b$ appears, missing its initial $1$ term,
in the OEIS \cite{oeis} as sequence \seqnum{A334820}.

All finite and infinite words in this paper are indexed starting at
position $1$.

\section{Basic properties}

\begin{proposition}
We have 
\begin{itemize}
\item[(a)] Let $\ell_i = |B_i|$.  Then $\ell_i = 2^{i-1} + i+1$ for $i \geq 1$.
\item[(b)] $B_i$ ends in $01$ for $i \geq 1$.
\end{itemize}
\label{useful}
\end{proposition}

\begin{proof}
A trivial induction on $i$.
\end{proof}

Call an index $i$ ${\it good\/}$ if $B_i$ has a border of length $i$; that is,
if $B_i[1..i] = B_i[\ell_i - i + 1..\ell_i]$.  
\begin{lemma}
If $i$ is good, then $\ell_i$ is good.
\label{good}
\end{lemma}

\begin{proof}
Assume $i$ is good.  The definition of $B_{i+1}$ 
says that the length-$\ell_i$ suffix of $B_{i+1}$ is
$B_i[\ell_i - i + 1.. \ell_i] B_i[i+1..\ell_i]$.
Because $i$ is good, this is equal to
$B_i[1..i]B_i[i+1..\ell_i] = B_i$.  Thus $B_{i+1}$ ends with $B_i$.

We now argue by induction that $B_s$ ends with $B_i$ for all $s> i$.
We just proved the base case of $s=1$.
Otherwise, note that
$B_{s+1} = B_s C_s$ and
$|C_s| = \ell_s - s = 2^{s-1} + 1 \geq 2^i + 1 \geq 2^{i-1} + i+1 = \ell_i$,
where we have used the fact that $2^{i-1} \geq i$ for all $i\geq 1 $.
Thus $C_s = B_s[s+1..\ell_s]$ ends with $B_i$, and hence so does $B_{s+1}$.

Thus, in particular, $B_{\ell_i}$ ends with $B_i$, and trivially
$B_{\ell_i}$ begins with $B_i$.  Thus $\ell_i$ is good.
\end{proof}

\begin{theorem}
The infinite word $\bf b$ is recurrent.
\label{recur}
\end{theorem}

\begin{proof}
Define $q_0 = 1$ and $q_{i+1} = \ell_{q_i}$ for $i \geq 0$.
By Lemma~\ref{good} every $q_i$ is good.

Let $w$ be a finite factor of $\bf b$.   Then $w$ occurs in some finite
prefix $B_n$ of $\bf b$.  Choose $i$ such that $q_i \geq n$.  
Since $B_n$ is a prefix of $B_{q_i}$, the word $w$ occurs in $B_{q_i}$.
But $B_{q_i}$ occurs twice in $B_{q_{i+1}}$, once as a prefix and once
as a suffix.  Hence $w$ occurs at least twice in $B_{q_{i+1}}$ and hence
twice in $\bf b$.  By a standard lemma, $\bf b$ is recurrent.
\end{proof}

So the infinite word $\bf b$ is recurrent,
but it is not uniformly recurrent.  In fact, as
we will see, $1^n$ occurs in $\bf b$ for all $n$. 

\begin{theorem}
Let $r_n$ be the starting position of the first occurrence of
$1^n$ in $\bf b$.  Then
$r_1 = 1$, $r_2 = 5$, and $r_{n+1} = 2^{r_n - 2} + r_n$ for
$n \geq 2$.   
\end{theorem}

\begin{proof}
We prove this by induction on $n$.  Since $B_2 = 10101101$,
it is clear that $r_1 = 1$ and $r_2 = 5$.  Assume $n \geq 2$.
Let us consider how $1^{n+1}$ can be formed.  Say it appears for the first
time in $B_{i+1}$.  Since
$B_{i+1} = B_i C_i$ and $C_i$ is a suffix
of $B_i$, if this is the first occurrence of $1^{n+1}$, it can only
be because $1^{n+1}$ straddles the boundary between $B_i$ and
$C_i$.  
But $B_i$ ends in $01$ from Proposition~\ref{useful} (b),
so we must have $C_i$ begins with $1^n$.  By induction
the first occurrence of $1^n$ is at position $r_n$ of $\bf b$.
It follows that $i+1 = r_n$, and 
$r_{n+1} = \ell_i = 2^{r_n - 1} + r_n + 1$, as desired.
\end{proof}

To describe how quickly $r_n$ grows, we use the tetration notation
of Knuth \cite{Knuth:1976}:  $a\tet 0 = 1$, and $a\tet n = a^{a \tet (n-1)}$ for
$n \geq 1$.  

\begin{corollary}
We have $r_n \geq  2 \tet (n-1) + 3$ for $n \geq 3$.  
\end{corollary}

\begin{proof}
By induction on $n$.  For $n = 2$ we have
$r_2 = 5 = 2 \tet 1 + 3$.  

Now assume $n \geq 2$ and $r_n \geq 2 \tet (n-1) + 3$.  
Then $r_n - 2 \geq 2 \tet (n-1)  + 1$, so
$2^{r_n - 2} \geq 2^{2 \tet (n-1) + 1} > 2^{2 \tet (n-1)} = 2 \tet n$.
Thus $r_{n+1} = 2^{r_n - 2} + r_n > 2 \tet n + r_n \geq 2 \tet n + 3$.
\end{proof}

\begin{remark}
We have $r_3 = 13$, $r_4 = 2061$, and $r_5 = 2^{2059} + 2061$.
\end{remark}

Next, we characterize all the factors of $\bf b$.

\begin{theorem}
A binary word is a factor of $\bf b$ if and only if
it contains no two consecutive $0$'s.  
\end{theorem}

\begin{proof}
Let us prove by induction on $i$ that no factor of $B_i$
contains two consecutive $0$'s.  This is clearly true for
$B_1$.  Otherwise, assume the claim is true for $i$.
Then $B_{i+1} = B_i C_i$, where $C_i$ is a suffix of $B_i$.
Every factor of $B_{i+1}$ is therefore a factor of $B_i$
or $C_i$, and hence contains no $00$ by induction, or
straddles the boundary.  But $B_i$ ends in $01$ by
Proposition~\ref{useful} (a), so no matter what $C_i$ begins
with, no factor containing $00$ can be formed.

Now let us prove by induction on $n$ that every word of
length $n$ without two consecutive $0$'s appears as a factor
of $\bf b$.  This is clearly true for $n = 1,2$, since
$B_2 = 10101101$.    Assume the claim is true for $n \geq 2$;
we prove it for $n+1$.  Let $w$ be a word of length $n+1$ without
two consecutive $0$'s.
Then either $w = 01w'$ or $w = 1w'$, where $w'$ does not contain
two consecutive $0$'s; by induction $w'$ appears in $\bf b$, say
at position $r$.  Since from Theorem~\ref{recur} we know that
$\bf b$ is recurrent, we can assume without loss of generality
that $r \geq 2$ and $\ell_{r-1} \geq n$.
Now consider $B_r = B_{r-1} B_i[r..\ell_{r-1}]$.
Since $B_{r-1}$ ends in $01$, we have $01w'$ and $1w'$ are both
in $B_r$.  
\end{proof}

\begin{corollary}
There are $F_{n+2}$ distinct factors of length $n$ in $\bf b$,
where $F_0 = 0$, $F_1 = 1$, $F_n = F_{n-1} + F_{n-2}$ are the
Fibonacci numbers.
\label{corfac}
\end{corollary}

Recall that an infinite word is said to be {\it morphic} if
it is the image, under a coding, of a fixed point of a morphism.
See, for example, \cite{Allouche&Shallit:2003}.

\begin{corollary}
The sequence $\bf b$ is not morphic.
\end{corollary}

\begin{proof}
This follows from a result of Ehrenfeucht, Lee, and Rozenberg
\cite{Ehrenfeucht&Lee&Rozenberg:1975} that if a sequence
is morphic, then its factor complexity is $O(n^2)$.
\end{proof}

\section{Density of $1$'s in {\bf b}}

Let ${\bf b} = b_1 b_2 b_3 \cdots$.
and $a(n)$ be the number of $1$'s appearing in ${\bf b}[1..n]$.
In this section we prove

\begin{theorem}
The limit $\alpha = \lim_{n \rightarrow \infty} a(n)/n$ exists, and
equals $2 - 2 \sum_{i \geq 1} b_i 2^{-i}$.
\end{theorem}

\begin{proof}
First we study the number $s(n)$ of $1$'s in prefixes of length
$\ell_n = |B_n|$.

Since $B_n = B_{n-1} B_{n-1}[n..\ell{n-1}]$, the number of $1$'s
the second factor is $s(n-1) - a(n-1)$.
Hence $s(n) = 2s(n-1) - a(n-1)$.  Dividing by $2^{n-1}$, we get
$$ {{s(n)}\over {2^{n-1}}} = {{s(n-1)} \over {2^{n-2}}} - 
{{a(n-1)} \over 2^{n-1}} .$$
Since $s(1) = 2$, iteration gives
\begin{equation}
 {{s(n)} \over {2^{n-1}}} = 2 - \sum_{1 \leq j < n} {{a(j)} \over {2^j}} .
\label{sn}
\end{equation}

Define $\alpha = 2 - \sum_{j \geq 1} {{a(j)} \over {2^j}} $.  This
is well-defined because $a(j) \leq j$ and
$\sum_{j \geq 1} j/2^j$ converges.   Then Eq.~\eqref{sn}
gives $\lim_{n \rightarrow \infty} {{s(n)}\over {2^{n-1}}} = \alpha$.
Since $\ell_n = 2^{n-1} + O(n)$, we also get
$\lim_{n \rightarrow \infty} {{s(n)}\over {\ell_n}} = \alpha$.

Next, we obtain an expression for $\alpha$.
Since $a(j) = \sum_{1 \leq i \leq j} b_i$,
we have
\begin{align*}
\sum_{j \geq 1} {{a(j)} \over {2^j}} &= \sum_{j\geq 1} 2^{-j} \sum_{1 \leq i \leq j} b_i \\
&= \sum_{i \geq 1} b_i \sum_{j \geq i} 2^{-j} \\
&= \sum_{i \geq 1} b_i 2^{1-i}  = 2 \sum_{i \geq 1} b_i 2^{-i}.
\end{align*}
Now if $\beta = \sum_{i \geq 1} b_i 2^{-i}$  then
\begin{equation}
\alpha = 2 - \sum_{j \geq 1} {{a(j)}\over {2^j}} = 2-2 \beta.
\label{alphabeta}
\end{equation}

We have now proven that the limiting density of 
$1$'s in $\bf b$ exists and equals
$\alpha$, if we only consider prefixes of length $\ell_i$.
It remains to consider arbitrary prefixes.

Define $E(N) = a(N) - \alpha N$ for $N \geq 1$.
From Eq.~\eqref{sn} we have
$$ {{s(n)} \over 2^{n-1}} = \alpha + \sum_{j \geq n} {{a(j)} \over
{2^j}} $$
and 
$$ \sum_{j \geq n} {{a(n)} \over {2^j}} = O(n 2^{-n}).$$
Therefore
$s(n) = \alpha 2^{n-1} + O(n)$.
Since $\ell_n = 2^{n-1} + n + 1$, it follows that
$$E(\ell_n) = s(n) - \alpha \ell_n = O(n).$$

Now consider an arbitrary integer $N \geq 1$.  
Write $N$ uniquely as $N = \ell_n + t$ where
$0 \leq t \leq \ell_{n+1} - \ell_n$.
Now
$\ell_{n+1} - \ell_n = \ell_n - n$, so 
$0 \leq t \leq \ell_n - n$.
Since $B_{n+1} = B_n B_n[n..\ell_n]$,
the length-$(\ell_n + t)$ prefix of $\bf b$
is $B_n B_n[n+1..n+t]$.  
Thus we get
$ a(N) = a(\ell_n + t) = s(n) + a(n+t)-a(n)$.
Subtracting $\alpha(\ell_n + t)$, we get
\begin{align}
E(\ell_n + t) &= s(n) - \alpha \ell_n + a(n+t) - a(n) - \alpha t  \nonumber \\
&= E(\ell_n) + E(n+t) - E(n)  \label{ee}.
\end{align}

Now define $\rho = \limsup_{N \rightarrow \infty} |E(N)|/N$.
From  the definition of $E$ we have $\rho < \infty$.

For $N = \ell_n + t$ we get from Eq.~\eqref{ee} that
$|E(N)| \leq |E(\ell_n)| + |E(n+t)| + |E(n)| $.
Dividing by $N = \ell_n + t$ we get
\begin{equation}
{{|E(N)|} \over N} \leq 
{{|E(\ell_n)|} \over N} + 
{{|E(n+t)|} \over {n+t}} \cdot {{n+t} \over {\ell_n + t}} +
{{|E(n)|} \over N}. \label{ineq}
\end{equation}
The first term tends to $0$ because
$|E(\ell_n)| = O(n)$.
The third term also tends to $0$
because $|E(n)| = O(n)$.

For the middle term,  note that 
$$ {{n+t} \over {\ell_n + t}} \leq {{\ell_n} \over {2\ell_n - n}} $$
because $0 \leq t \leq \ell_n - n$.
Now $(n+t)/(\ell_n + t)$ is an increasing function of $t$,
and its maximum on the interval $t \in [0, \ell_n - n]$
is attained at $t = \ell_n -n$, where it takes the value
$\ell_n/(2\ell_n - n)$.
Therefore this maximum is $1/2 + o(1)$.  

Taking $\limsup$ in \eqref{ineq}, we get $\rho \leq \rho/2$,
and hence $\rho = 0$.   Thus
$\lim_{N \rightarrow \infty} E(N)/N = 0$
and hence
$\lim_{N \rightarrow \infty} a(N)/N = \alpha$.
\end{proof}

\begin{remark}
The density  $\alpha$ is approximately $0.64505878493452$.
\end{remark}

\section{More about the density $\alpha$}

\begin{theorem}
The number $\alpha$ is transcendental.
\end{theorem}

\begin{proof}
From Eq.~\eqref{alphabeta}, it suffices to prove that 
$\beta = \sum_{i \geq 1} b_i 2^{-i}$ is transcendental.

For this we use the following simplified version of the
``stammering'' criterion of Adamczewski and Bugeaud
\cite[Thm.~5]{Adamczewski&Bugeaud:2007}:
\begin{lemma}
Let ${\bf a} = (a_i)_{i \geq 1}$ be a bounded sequence of integers.
If there exists a real number $t>1$  and two sequences
of finite words $(u_n)_{n \geq 1}$ and $(v_n)_{n \geq 1}$  such that
\begin{itemize}
\item[(a)] $u_n v_n^t$ is a prefix of $\bf a$ for $n \geq 1$;
\item[(b)] The sequence $(|u_n|/|v_n|)_{n \geq 1}$ is bounded above
by a constant;
\item[(c)] The sequence $(|v_n|)_{n \geq 1}$ is increasing;
\end{itemize}
then $\sum_{i\geq 1} a_i 2^{-i}$ is either rational
or transcendental.
\label{adam}
\end{lemma}

Take ${\bf a} = {\bf b}$, $u_n = {\bf b}[1..n]$,
$v_n = {\bf b}[n+1..\ell_n]$, and $t = 2$ in Lemma~\ref{adam}.
Then $B_n = u_n v_n$ and $B_{n+1} = u_n v_n^2$, so
$u_n v_n^t$ is a prefix of $\bf b$.  Also
$|u_n| = n$, $|v_n| = \ell_n - n = 2^{n-1} + 1$,
so $|u_n|/|v_n|$ is bounded, and $|v_n| \rightarrow \infty$.

Thus by Lemma~\ref{adam}, the number
$\beta$ is either rational or transcendental.
But it cannot be rational, for then the factor complexity
of $\bf b$ would be $O(1)$, contradicting
Corollary~\ref{corfac}.  Thus $\beta$ is transcendental
and so is $\alpha$.
\end{proof}

\section{Acknowledgments}

These results were obtained during multi-day experimentation with  
APL and collaboration with ChatGPT 5.5,
in May 2026.  All claims were rewritten and verified by the author.

\end{document}